\documentclass[11pt,oneside,english,jou, a4paper]{amsart}
\usepackage{a4wide}
\usepackage[T1]{fontenc}
\usepackage[latin9]{inputenc}
\usepackage{babel}
\usepackage{verbatim}
\usepackage{mathrsfs}
\usepackage{dsfont}
\usepackage{amstext}
\usepackage{amsthm}
\usepackage{amssymb}
\usepackage{esint}
\usepackage{enumerate}
\usepackage[margin=2.3cm]{geometry}
\usepackage[title]{appendix}

\makeatletter
\numberwithin{equation}{section}
\numberwithin{figure}{section}
\theoremstyle{plain}
\newtheorem{thm}{\protect\theoremname}[section]
\theoremstyle{definition}
\newtheorem{rem}[thm]{\protect\remarkname}
\theoremstyle{definition}
\newtheorem{defn}[thm]{\protect\definitionname}
\theoremstyle{plain}
\newtheorem{prop}[thm]{\protect\propositionname}
\theoremstyle{plain}
\newtheorem{lem}[thm]{\protect\lemmaname}
\theoremstyle{plain}

\theoremstyle{plain}
\newtheorem{cor}[thm]{\protect\corollaryname}

\theoremstyle{definition}
\newtheorem{problem}{\protect\problemname}
\theoremstyle{definition}

\theoremstyle{definition}

\theoremstyle{definition}

\theoremstyle{definition}

\usepackage{tikz}
\usetikzlibrary{shapes,arrows}

\newcommand{\eps}{\varepsilon}
\newcommand{\R}{\mathbb{R}}
\newcommand{\N}{\mathbb{N}}
\newcommand{\Z}{\mathbb{Z}}
\newcommand{\mE}{\mathcal{E}}\newcommand{\mS}{\mathcal{S}}

\newcommand{\mX}{\mathcal{X}}
\newcommand{\mP}{\mathcal{P}}
\newcommand{\mU}{\mathcal{U}}
\newcommand{\bp}{\begin{proof}}
	\newcommand{\ep}{\end{proof}}
\newcommand{\ummdim}{\overline{\mdim}_M}
\newcommand{\lmmdim}{\underline{\mdim}_M}
\newcommand{\mmdim}{\mdim_M}

\DeclareMathOperator{\diam}{diam}
\DeclareMathOperator{\mdim}{mdim}

\newcommand{\mBKe}{\mathrm{mBKe}}
\DeclareMathOperator{\supp}{supp}

\DeclareMathOperator{\Leb}{Leb}

\makeatother

\providecommand{\conjecturename}{Conjecture}
\providecommand{\corollaryname}{Corollary}
\providecommand{\definitionname}{Definition}
\providecommand{\examplename}{Example}
\providecommand{\lemmaname}{Lemma}
\providecommand{\problemname}{Problem}
\providecommand{\propositionname}{Proposition}
\providecommand{\remarkname}{Remark}
\providecommand{\theoremname}{Theorem}
\providecommand{\taskname}{Task}
\providecommand{\pytaniename}{Pytanie}
\providecommand{\trudnosciname}{Trudnoci}
\begin{document}
\sloppy
\title{Around the variational principle for metric mean dimension}

\author[Y. Gutman]{Yonatan Gutman$^1$}
\address{$^1$Institute of Mathematics, Polish Academy of Sciences,
	ul. \'Sniadeckich 8, 00-656 Warszawa, Poland}
\email{y.gutman@impan.pl}

\author[A. \'{S}piewak]{Adam \'{S}piewak$^2$}
\address{$^2$Institute of Mathematics, University of Warsaw, ul. Banacha 2, 02-097 Warszawa, Poland}
\email{a.spiewak@mimuw.edu.pl}

\begin{abstract}
We study variational principles for metric mean dimension. First we prove that in the variational principle of Lindenstrauss and Tsukamoto it suffices to take supremum over ergodic measures.  Second we derive a variational principle for metric mean dimension involving growth rates of measure-theoretic entropy of partitions decreasing in diameter which holds in full generality and in particular does not necessitate the assumption of tame growth of covering numbers. The expressions involved are a dynamical version of R\'{e}nyi information dimension. Third we derive a new expression for Geiger-Koch information dimension rate for ergodic shift-invariant measures. Finally we develop a lower bound for metric mean dimension in terms of Brin-Katok local entropy.
\end{abstract}



\thanks{Y.G and A.\'S were partially supported by the National Science Center (Poland) grant 2016/22/E/ST1/00448.} 

\maketitle

\section{Introduction}\label{sec:intro}

\subsection{General background}

This paper contributes to the study of ergodic aspects of metric mean dimension - an invariant of dynamical systems introduced by Lindenstrauss and Weiss \cite{LW00} and devised to quantify the complexity of infinite entropy systems. A link between ergodic theory and metric mean dimension was established recently by Lindenstrauss and Tsukamoto \cite{lindenstrauss_tsukamoto2017rate}, who proved a variational principle expressing metric mean dimension as a supremum of certain rate-distortion functions over the set of invariant measures of the system. The goal of this note is to further study variational principle of \cite{lindenstrauss_tsukamoto2017rate} as well as its ramifications in terms of measure-theoretic entropy and Brin-Katok local entropy. Specifically in Section \ref{sec:ergodic_measures} we prove that in the Lindenstrauss-Tsukamoto variational principle it is enough to take supremum over ergodic measures. In Section \ref{sec:mrid} we derive a variational principle for metric mean dimension involving growth rates of measure-theoretic entropy of partitions decreasing in diameter.
The expressions involved are a dynamical version of the Pesin's definition of the so-called R\'{e}nyi information dimension for measures in general metric spaces. The main ingredient of the proof is the local variational principle for the entropy of a fixed open cover. In Section \ref{sec:single_mrid} we introduce mean R\'enyi information dimension of an invariant measure and show that it equals Geiger-Koch information dimension rate in the context of stochastic processes. Finally in Section \ref{sec:brin_katok} we introduce mean Brin-Katok local entropy of a dynamical system and prove that it gives a lower bound for metric mean dimension.

\subsection{Metric mean dimension}\label{sec:mmdim}

By a \emph{topological dynamical system} we understand a triple $(\mX, \rho, T)$ consisting of a compact metric space $(\mX, \rho)$ and a homeomorphism $T : \mX \to \mX$.

\begin{defn}\label{def:hashtag} Let $(\mX, \rho)$ be a compact metric space. For $\eps > 0$, the \emph{$\eps$-covering number} of a subset $A \subset \mX$, denoted by $\#(A,\rho,\varepsilon)$, is the minimal cardinality $N$ of an open cover $\{U_{1},\dots,U_{N}\}$ of $A$ such that all $U_{n}$ have diameter at most $\varepsilon$.
\end{defn}

\begin{defn}\label{d:mmdim}
	\label{def:d_n}Let $(\mX, \rho)$ be a compact metric space and let $T:\mathcal{X}\to \mathcal{X}$
	be a homeomorphism. For $n\in\N$ define a metric $\rho_{n}$ on $\mathcal{X}$
	by $\rho_{n}(x,y)=\max \limits_{0\leq k<n}\rho(T^{k}x,T^{k}y)$.
	The upper and lower \emph{metric mean dimensions} of the system $(\mX, \rho, T)$ are defined as
	\[
	\overline{\mdim}_{M}(\mX, \rho, T)=\limsup_{\eps\to0} \frac{S(\mX, \rho, T, \eps)}{\log\frac{1}{\eps}}
	\]
	and	
	\[
	\underline{\mdim}_{M}(\mX, \rho, T)=\liminf_{\eps\to0} \frac{S(\mX, \rho, T, \eps)}{\log\frac{1}{\eps}},
	\]
	where
	\[S(\mX, \rho, T, \eps) = \lim \limits_{n \to \infty} \frac{1}{n}\log\#(\mathcal{X},\rho_{n},\eps)\]
	and $\#(\mX, \rho_n, \eps)$ is the $\eps$-covering number of $\mX$ with respect to the metric $\rho_n$ (see Definition \ref{def:hashtag}). The limit defining $S(\mX, \rho, T, \eps)$  exists due to the subadditivity of the sequence $n\mapsto\log\#(\mathcal{X},\rho_{n},\eps)$. If the upper and lower limits coincide, then we call their common value the \emph{metric mean dimension} of $(\mX, \rho, T)$ and denote it by $\mmdim(\mX, \rho, T)$.
\end{defn}

Metric mean dimension was initially introduced in \cite{LW00} as an upper bound for topological mean dimension - a topological invariant of dynamical systems introduced by Gromov \cite{G} and studied by Lindenstrauss and Weiss \cite{LW00} in connection with embedding problems in topological dynamics. Further connections between embedding properties and topological mean dimension were studied by several authors \cite{LT12, GutLinTsu15, gutman2017application, GutTsu16}. Metric mean dimension turned out to be a useful tool in calculating topological mean dimension of certain systems (see e.g. \cite{TsYangMills18} and \cite{TsBrody18} for applications to dynamical systems arising from geometric analysis) as well as a meaningful quantity from the point of view of analog compression (see \cite{GS19short, GS20mmdim_compress}).

\subsection{The classical variational principle}

Let $\mX$ be a measurable space. Let $\mu$ be a probability measure on $\mX$ and let $P = \{A_1, ..., A_N\}$ be a measurable partition of $\mX$. The \emph{entropy} $H_{\mu}(P)$ of $P$ with respect to the measure $\mu$ is defined as
\[ H_{\mu}(P) = - \sum \limits_{A \in P} \mu(A) \log \mu(A) \]
(with the convention $0 \cdot \infty = 0$). For a topological dynamical system $(\mX, \rho, T)$, we say that a Borel probability measure $\mu$ on $\mX$ is $T$\emph{-invariant} if $\mu(A) = \mu(T^{-1}A)$ for all Borel sets $A \subset \mX$. Let $\mP_{T}(\mathcal{X})$ denote the set of all such measures. For a given $\mu\in \mP_{T}(\mathcal{X})$ and a finite Borel partition $P$ of $\mX$, the \emph{(dynamical) entropy} of $P$ is defined as
\[ h_{\mu}(P) = \lim \limits_{n \to \infty} \frac{H_{\mu} (P \vee T^{-1}P \vee ... \vee T^{-(n-1)}P)}{n},\]
where $P \vee Q = \{ A \cap B : A\in P, B \in Q \}$ denotes the join of partitions $P$ and $Q$ and $T^{-1}P = \{ T^{-1}A : A \in P \}$ denotes the preimage of partition $P$. The \emph{measure-theoretic (Kolmogorov-Sinai) entropy} $h_{\mu}(T)$ of $\mu$ is defined as
\[ h_{\mu}(T) = \sup \{ h_\mu(P) : P \text{ is a finite Borel partition of } \mX \}. \]
The \emph{topological entropy} $h_\mathrm{top}(\mX,T)$ of the topological dynamical system $(\mX,\rho, T)$ is defined as
\[ h_\mathrm{top}(\mX,T) = \lim \limits_{\eps \to 0} S(\mX, \rho, T,\eps) = \lim \limits_{\eps \to 0} \lim \limits_{n \to \infty} \frac{\log\#(\mathcal{X},\rho_{n},\eps)}{n} \]
(with $\rho_n$ as in Definition \ref{d:mmdim}).

One of the basic results of ergodic theory is the variational principle, connecting topological entropy with measure-theoretic entropy, proved originally in \cite{goodman1971relating}:

\begin{thm}\label{thm:entropy_var_prin}
	Let $(\mX, \rho, T)$ be a topological dynamical system. Then
	\[h_\mathrm{top}(\mX,T) = \sup \{ h_{\mu}(T) : \mu \in \mP_{T}(\mX) \}.\]
\end{thm}

For more on entropy theory of dynamical systems see \cite{D11}. The main goal of this paper is to study similar formulas expressing metric mean dimension in terms of measure-theoretic invariants of a system.

\subsection{The Lindenstrauss-Tsukaomoto variational principle}

In \cite{lindenstrauss_tsukamoto2017rate}, Lindenstrauss and Tsukamoto provided a variational principle for metric mean dimension in terms of the following rate-distortion function:

\begin{defn}\label{def:rate_dist_LT}
	Let $(\mX, \rho, T)$ be a topological dynamical system and $\mu\in \mP_{T}(\mathcal{X})$. For $ p \in [1, \infty),\ n\in\mathbb{{N}}$ and $\eps>0$ define $R_{\mu,p}(\varepsilon,n)$ as the infimum
	of
	\[\frac{I(X;Y)}{n},\]
	where $X$ and $Y=(Y_{0},\dots,Y_{n-1})$ are random variables defined
	on some probability space $(\Omega,\mathbb{P})$ such that
	
	\begin{itemize}
		\item $X$ takes values in $\mathcal{X}$, and its law is given by $\mu$.
		\item Each $Y_{k}$ takes values in $\mathcal{X}$, and $Y$ approximates
		the process $(X,TX,\dots,T^{n-1}X)$ in the sense that
		\begin{equation}\label{eq: distortion condition-1}
		\mathbb{E}\left(\frac{1}{n}\sum_{k=0}^{n-1}\rho(T^{k}X,Y_{k})^p\right)<\varepsilon^p.
		\end{equation}
	\end{itemize}
	Here $\mathbb{E}(\cdot)$ is the expectation with respect to the probability
	measure $\mathbb{P}$ and $I(X;Y)$ is the mutual information of random vectors $X$ and $Y$ (see \cite{Gray11} and \cite{lindenstrauss_tsukamoto2017rate}). As the sequence $n \mapsto n R_{\mu,p}(\eps, n)$ is subadditive (which follows as in \cite[Theorem 9.6.1]{Gallager68}), we can make the following definition:
	\[R_{\mu,p}(\varepsilon)=\lim_{n \to \infty} R_{\mu,p}(\varepsilon,n) = \inf_{n \in \N}R_{\mu,p}(\varepsilon,n).\]
	The function $\eps \to R_{\mu,p}(\eps)$ is called the $L^p$ \emph{rate-distortion function} of $\mu$.
\end{defn}

\begin{rem}\label{rem:rd_function_finite}
	As pointed out in \cite[Remark IV.3]{lindenstrauss_tsukamoto2017rate}, in Definition \ref{def:rate_dist_LT} it is enough to consider random vectors $Y$ taking finitely many values. As $I(X;Y) \leq H(Y) \leq \infty$ for such $Y$, we obtain also $R_{\mu, p}(\eps) < \infty$ for every $\eps>0$.
\end{rem}

Note that the condition (\ref{eq: distortion condition-1}) above is given in terms of the average distance $\frac{1}{n} \sum \limits_{k=0}^{n-1}\rho(T^k x, T^k y)^p$, while metric mean dimension is defined via the maximum metric $\rho_n$. In order to compare these two notions, the following condition is introduced in \cite{lindenstrauss_tsukamoto2017rate}:

\begin{defn}\label{def:tame_growth}(\cite[Condition II.3]{lindenstrauss_tsukamoto2017rate}) Let $(\mX, \rho)$ be a compact metric space. It is said to have the \emph{tame growth of covering numbers} if for every $\delta>0$ we have
	\[ \lim \limits_{\eps \to 0} \eps^{\delta} \log \# (\mX, \rho, \eps) = 0. \]
\end{defn}

Lindentrauss and Tsukamoto proved the following variational principle for metric mean dimension:

\begin{thm}\label{thm:lin_tsu_var_prin}(\cite[Corollary III.6]{lindenstrauss_tsukamoto2017rate}) Let $(\mX, \rho)$ be a compact metric space having the tame growth of covering numbers and let $T : \mX \to \mX$ be a homeomorphism. Then, for any $p \in [1, \infty)$
	\[ \ummdim(\mX, \rho, T) = \limsup \limits_{\eps \to 0} \sup_{\mu \in \mP_{T}(\mX)} \frac{ R_{\mu, p}(\eps)}{\log \frac{1}{\eps}}\]
	and
	\[ \lmmdim(\mX, \rho, T) = \liminf \limits_{\eps \to 0} \sup_{\mu \in \mP_{T}(\mX)} \frac{ R_{\mu, p}(\eps)}{\log \frac{1}{\eps}}.\]
\end{thm}

Velozo and Velozo \cite{velozo2017rate} provided an alternative formulation in terms of Katok's entropy. For an extension of Theorem \ref{thm:lin_tsu_var_prin} to actions of countable discrete amenable groups see \cite{CDZ17}. See \cite{LinTsuDouble19} for a double (minimax) variational principle, expressing topological mean dimension as a supremum of rate-distortion dimensions of invariant measures, minimized over metrics compatible with the topology of the phase space. See also \cite{TsuDoublePot20} for an extension of this result to systems with potential.

\section{The variational principle with respect to ergodic measures}\label{sec:ergodic_measures}
A measure $\mu \in \mP_{T}(\mX)$ is called \emph{ergodic} if every $T$-invariant set (i.e. set $A$ satisfying $T^{-1}A=A$) is either of full or zero $\mu$-measure. Let $\mathcal{E}_T(\mX)$ denote the set of ergodic measures. It is known that in the variational principle for topological entropy (Theorem \ref{thm:entropy_var_prin}), it is enough to take supremum over ergodic measures (see e.g. \cite[Theorem 6.8.1]{D11}). Our first result states that this is also the case in the Lindenstrauss-Tsukamoto variational principle.
\begin{thm}
	\label{thm:ergodic_var_prin} Let $(\mX, \rho, T)$ be a topological
	dynamical system with $(\mX, \rho)$ having the tame growth of covering numbers.
	Then for $p \in [1, \infty)$
	\[\overline{\mdim}_{M}(\mX, \rho, T)= \limsup\limits_{\varepsilon\to0}\sup\limits_{\mu\in\mathcal{E}_{T}(\mathcal{X})}\frac{R_{\mu,p}(\eps)}{\log\frac{1}{\varepsilon}}.\]
	and
	\[\underline{\mdim}_{M}(\mX, \rho, T)= \liminf\limits_{\varepsilon\to0}\sup\limits_{\mu\in\mathcal{E}_{T}(\mathcal{X})}\frac{R_{\mu,p}(\eps)}{\log\frac{1}{\varepsilon}}.\]
\end{thm}

We decompose the proof of Theorem \ref{thm:ergodic_var_prin} into several lemmas. Our main tool will be Corollary 1 of \cite{ECG94}. Since it is stated for stationary stochastic processes with values in a given alphabet space $A$ (possibly uncountable), we will begin in the setting of shift-invariant measures. Let $(A,d)$ be a compact metric space and let $\sigma:A^\Z \to A^\Z$ be the shift transformation, i.e.
\[ \sigma((x_i)_{i\in \Z}) =(x_{i+1})_{i\in \Z}.\]
Note that the set of distributions of stationary stochastic processes with values in $A$ corresponds to the set $\mP_{\sigma}(A^\Z)$ of shift-invariant measures on $A^\Z$. Let $\pi_n : A^\Z \to A^n$ denote the projection $\pi_n((x_i)_{i \in \Z}) = (x_0, \ldots x_n)$. In this setting it is natural to consider a slight modification of the rate-distortion function from Definition \ref{def:rate_dist_LT}, more commonly used in information theory (see e.g. \cite[Chapter 9]{Gray11}).

\begin{defn}\label{def:rate_distortion_function} Let $(A,d)$ be a compact metric space and let $\mu\in  \mP_{\sigma}(A^\Z)$. For $p \in [1, \infty),\ \varepsilon>0$ and $n\in\mathbb{{N}}$
	we define the \emph{$L^p$ rate-distortion function} $\tilde{R}_{\mu,p}(n,\varepsilon)$
	as the infimum of
	\begin{equation}
	\frac{I(X;Y)}{n},\label{eq: definition of rate-distortion function}
	\end{equation}
	where $X=(X_0, ..., X_{n-1})$ and $Y=(Y_{0},\dots,Y_{n-1})$ are random variables defined on some probability space $(\Omega,\mathbb{P})$ such that

	\begin{itemize}
		\item $X=(X_0, ..., X_{n-1})$ takes values in $A^n$, and
		its law is given by $(\pi_n)_*\mu$.
		\item $Y=(Y_{0},\dots,Y_{n-1})$ takes values in $A^n$ and approximates $(X_{0},X_{1},\dots,X_{n-1})$
		in the sense that
		
		\begin{equation}\label{eq: distortion condition}
		\begin{gathered}
		\mathbb{{E}}\left(\frac{1}{n}\sum_{k=0}^{n-1}d(X_{k},Y_{k})^p\right)\leq\varepsilon^p.
		\end{gathered}
		\end{equation}
	\end{itemize}
	Again, the function $n \mapsto n\tilde{R}_{\mu, p}(n, \eps)$ is subadditive (as in \cite[Theorem 9.6.1]{Gallager68}). Hence, we can define
	\[
	\tilde{R}_{\mu,p}(\eps)=\lim_{n \to \infty}\tilde{R}_{\mu,p}(n,\varepsilon) = \inf_{n \in \N}\tilde{R}_{\mu,p}(n,\varepsilon).
	\]
	We set $\tilde{R}_{\mu}(\eps):= \tilde{R}_{\mu, 1}(\eps)$.
	
\end{defn}

\begin{rem}\label{rem:rate_distortion_properties}
	Similarly to \cite[Remark IV.3]{lindenstrauss_tsukamoto2017rate}, in Definition \ref{def:rate_distortion_function} it is enough to consider random vectors $Y$ taking only finitely many values. As $I(X;Y) \leq H(Y) < \infty$ for $Y$ taking finitely many values, we obtain that $\tilde{R}_{\mu,p}(\eps) < \infty$ for every $\eps>0$ (see also Remark \ref{rem:rd_function_finite}).
	Note that $\tilde{R}_{\mu, 1}(\eps) \leq \tilde{R}_{\mu,p}(\eps)$ for $p \in [1, \infty)$ and $\eps \geq 0$. The function $(0, \infty) \ni \eps \mapsto \tilde{R}_{\mu, p}(\eps)$ is nonnegative, nonincreasing, convex and continuous (see \cite[Lemma 9.5]{Gray11}). Moreover, $\tilde{R}_{\mu, p}(\eps) = 0$ for $\eps \geq \diam(A)$, since for such $\eps$ taking $Y \equiv const$ yields a random vector satisfying (\ref{eq: distortion condition}) and such that $I(X;Y) \leq H(Y) = 0$.  For more details on rate-distortion function see \cite{Gray11}.
\end{rem}

\begin{rem}
	Note that for the dynamical system $\mX = A^\Z,\ T = \sigma$, Definitions \ref{def:rate_dist_LT} and \ref{def:rate_distortion_function} are different, as the former requires a metric $\rho$ on the full shift space $A^\Z$, while the latter makes use of a metric $d$ on the alphabet space $A$ itself. However, for the natural choice
	\begin{equation}\label{eq:product_metric} \rho( (x_i)_{i\in \Z}, (y_i)_{i \in \Z}) = \sum \limits_{i \in \Z} \frac{d(x_i, y_i)}{2^{|i|}},
	\end{equation}
	the functions $R_{\mu}(\eps)$ (defined via $\rho$) and $\tilde{R}_{\mu}(\eps)$ (defined via $d$) are comparable (see \cite[Proposition C-B.1]{GS20mmdim_compress}), i.e. the inequality
	\[ R_{\mu, p}(14\eps) \leq \tilde{R}_{\mu, p}(\eps) \leq R_{\mu, p}(\eps) \]
	holds.
\end{rem}

Let us define now the distortion-rate function - the formal inverse of $\tilde{R}_{\mu}(\eps)$.

\begin{defn}\cite[Chapter 9]{Gray11}
	Let $(A,d)$ be a compact metric space and let $\mu \in \mP_{\sigma}(A^\Z)$. For $R>0,\ N \in \N$ and $p \in [1, \infty)$ define the \emph{$L^p$ distortion-rate function} $D_{\mu,p}(R, n)$ as the infimum of
	\[ \left(\frac{1}{n}\mathbb{E}\sum\limits_{i=0}^{n-1}d(X_{i},Y_{i})^p\right)^{\frac{1}{p}}, \]
	where $X=(X_0, ..., X_{n-1})$ and $Y=(Y_{0},\dots,Y_{n-1})$ are random variables defined on some probability space $(\Omega,\mathbb{P})$ such that
	\begin{itemize}
		\item $X=(X_0, ..., X_{n-1})$ takes values in $A^n$, and
		its law is given by $(\pi_n)_*\mu$.
		\item $Y=(Y_{0},\dots,Y_{n-1})$ takes values in $A^n$ and
		\[\frac{1}{n}I(X ; Y)\leq R.\]
	\end{itemize}
	Using the subadditivity of $n \mapsto nD_{\mu,p}(R,n)$ (see \cite[Lemma 9.2]{Gray11}), we define further
	\[D_{\mu, p}(R)= \lim\limits_{n \in \N} \ D_{\mu, p}(R, n) = \inf\limits_{n \in \N} \ D_{\mu, p}(R, n).\]
	Note that there always exists a pair of random vectors $X$ and $Y$ satisfying above conditions, as for $Y \equiv const$ it holds $\frac{1}{n}I(X;Y) \leq H(Y) = 0$.
\end{defn}
The basic properties of $D_{\mu,p}(R,n)$ and $D_{\mu,p}(R)$ are given by the following lemma.

\begin{lem}\label{lem:distortion_rate_properties}
	Let $(A,d)$ be a compact metric space and let $\mu \in \mP_{\sigma}(A^\Z)$. The functions
	\[(0, \infty) \ni R \mapsto D_{\mu,p}(R, n) \in [0, \infty)\ \text{ and }\ (0, \infty) \ni R \mapsto D_{\mu,p}(R) \in [0, \infty)\]
	are nonnegative, nonincreasing, convex, continuous and bounded by $\diam(A)$. Moreover
	\begin{equation}\label{e:distortion_rate_zero}\lim \limits_{R \to \infty}D_{\mu, p}(R) = 0
	\end{equation}
	and $D_{\mu, p}(R)$ is strictly decreasing on the open set $\{ R>0 : D_{\mu, p}(R) > 0 \}$. Its inverse on $\{ R>0 : D_{\mu, p}(R) > 0 \}$ is the function $\tilde{R}_{\mu, p}$, i.e. for $D_0>0,\ R_0 >0$ we have $D_{\mu, p}(R_0) = D_0$ if and only if $\tilde{R}_{\mu, p}(D_0) = R_0$.
\end{lem}

Most of the properties stated in the above lemma follow from \cite[Lemma 9.1]{Gray11}. For the convenience of the reader, we include the proof of the remaining points.

\begin{proof}
	Nonnegativity, convexity, continuity and nonincreasing follow from \cite[Lemma 9.1]{Gray11}. For the upper bound note that
	\[ \frac{1}{n}\mathbb{E}\sum\limits_{i=0}^{n-1}d(X_{i},Y_{i})^p \leq \diam(A)^p. \]
	For (\ref{e:distortion_rate_zero}), fix $\eps>0$ and let $P$ be a finite Borel partition of $A$ with diameter smaller than $\eps$. For each $B \in P$ choose a unique point $x_B \in B$ and let $f : A \to A$ be given by $f(x)=x_B$ if $x \in B$. Then for any $n \in \N$ and a random vector $X=(X_0, ..., X_{n-1})$ such that $X\sim(\pi_n)_*\mu$ and $Y = (Y_0, ..., Y_{n-1})$ given by $Y_i= f \circ X_i,\ i = 0, ..., n-1$ we have \[\frac{1}{n}\mathbb{E}\sum\limits_{i=0}^{n-1}d(X_{i},Y_{i})^p \leq \eps^p \text{ and }\frac{1}{n}I(X ; Y)\leq \frac{1}{n}H(Y) < \infty,\]
	hence for $R \geq \frac{1}{n}H(Y)$ it holds $D_{\mu, p}(R) \leq D_{\mu, p}(\frac{1}{n}H(Y)) \leq D_{\mu, p}(\frac{1}{n}H(Y),n) \leq \eps$. As $\eps>0$ was arbitrary, (\ref{e:distortion_rate_zero}) is proved. Convexity of $D_{\mu, p}$ together with $\lim \limits_{R \to \infty}D_{\mu,p}(R) = 0$ imply that $D_{\mu,p}(\cdot)$ is strictly decreasing on the open set $\{ R>0 : D_{\mu,p}(R) > 0 \}$. Let us show now that if for $D_0>0,\ R_0 >0$ we have $D_{\mu, p}(R_0) = D_0$, then $\tilde{R}_{\mu,p}(D_0) = R_0$. First, we show the inequality $\tilde{R}_{\mu,p}(D_0) \leq R_0$. By the strict monotonicity of $D_{\mu, p}(\cdot)$ we have $D_{\mu,p}(R_0 + \eta) < D_0$ for every $\eta > 0$. Consequently, for every $\eta > 0$ there exist $\xi > 0,\ n \in \N$ and random vectors  $X=(X_0, ..., X_{n-1}),\ Y=(Y_0, ..., Y_{n-1})$ such that
	\[X \sim (\pi_n)_* \mu,\ \frac{1}{n}I(X ; Y)\leq R_0 + \eta \text{ and } \frac{1}{n}\mathbb{E}\sum\limits_{i=0}^{n-1}d(X_{i},Y_{i})^p \leq (D_0 - \xi)^p.\]
	Therefore $\tilde{R}_{\mu,p}(D_0) \leq \tilde{R}_{\mu,p}(D_0 - \xi) \leq R_0 + \eta$. For the other direction assume that $\tilde{R}_{\mu,p}(D_0) < R_0$. Then there exists $\eta > 0,\ n \in \N$ and random vectors  $X=(X_0, ..., X_{n-1}),\ Y=(Y_0, ..., Y_{n-1})$ such that
	\[X \sim (\pi_n)_* \mu,\ \frac{1}{n}I(X ; Y)\leq R_0 - \eta \text{ and } \frac{1}{n}\mathbb{E}\sum\limits_{i=0}^{n-1}d(X_{i},Y_{i})^p \leq D_0^p.\]
	Consequently $D_{\mu, p}(R_0 - \eta) \leq D_0$, what contradicts the strict monotonicity of $D_{\mu,p}(\cdot)$. The opposite implication follows similarly (see Remark \ref{rem:rate_distortion_properties} for the properties of $\tilde{R}_{\mu, p}$).
\end{proof}

\begin{lem}\label{lem:distortion_rate_ergodic}\cite[Corollary 1]{ECG94}
	Let $(A,d)$ be a compact metric space and let $\mu \in \mP_{\sigma}(A^\Z)$. Let $\{ \mu_x : x \in A^\Z \}$ be the ergodic decomposition of the measure $\mu$. Then for $R \geq 0$
	\[ D_{\mu, p}(R) = \inf \bigg\{ \int \limits_{A^\Z} D_{\mu_x, p}(S(x))d\mu(x)\ \Big|\ S : A^\Z \to [0, \infty) \text{ is Borel measurable} \]
	\[ \text { and } \int \limits_{A^\Z} S(x)d\mu(x) \leq R \bigg\}. \]
	In particular
	\[ D_{\mu, p}(R) \leq \int \limits_{A^\Z}  D_{\mu_x, p}(R) d\mu(x). \]
\end{lem}

\begin{proof}
	The equality is the content of \cite[Corollary 1]{ECG94}. Applying it to the constant function $S(x) \equiv R$, we obtain $D_{\mu, p}(R) \leq \int \limits_{A^\Z}  D_{\mu_x, p}(R) d\mu(x)$.
\end{proof}

By a \emph{subshift} we will understand a subset $\mS \subset A^\Z$ which is closed in the product topology and shift-invariant. Note that $(\mS, \rho, \sigma)$ is itself a topological dynamical system with the metric $\rho$ defined in (\ref{eq:product_metric}), which metrizes the product topology on $A^\Z$.

\begin{lem}\label{lem:rate_distortion_ergodic}
	Let $(A,d)$ be a compact metric space, let $\mS \subset A^\Z$ be a subshift and let $\mu \in \mP_{\sigma}(\mS)$. For every $p \in [1, \infty)$ and $D > 0$ there exists an ergodic measure $\nu \in \mE_{\sigma}(\mS)$ such that $\tilde{R}_{\mu, p}(D) \leq \tilde{R}_{\nu, p}(D)$.
\end{lem}

\begin{proof}
	Fix $D > 0$. Assume first that for each $R>0$ we have $D_{\mu, p}(R) < D$. By the definition of $D_{\mu, p}(R)$, this means that for every $R > 0$ there exist $n \in \N$ and random vectors $X=(X_0, ..., X_{n-1}),\ Y=(Y_0, ..., Y_{n-1})$ such that
	\[X \sim (\pi_n)_* \mu,\ \frac{1}{n}I(X ; Y)\leq R \text{ and } \frac{1}{n}\mathbb{E}\sum\limits_{i=0}^{n-1}d(X_{i},Y_{i})^p \leq D^p.\]
	Therefore $R_{\mu, p}(D) = 0$ and there is nothing to prove. Assume now that there exists $R_0>0$ such that $D_{\mu, p}(R_0) = D$. By Lemma \ref{lem:distortion_rate_properties} we have $R_{\mu, p}(D) = R_0$. By Lemma \ref{lem:distortion_rate_ergodic} we get
	\[D_{\mu, p}(R_{0})\leq\int\limits_{\mathcal{X}}D_{\mu_x,p}(R_0)d\mu(x),\]
	where $\{ \mu_x : x \in \mS  \}$ is the ergodic decomposition of $\mu$.
	In particular, there exists $x\in\mS$ such that
	$D = D_{\mu, p}(R_{0})\leq D_{\mu_x, p}(R_0)$. By continuity and monotonicity there exists $\tilde{R} > 0$ such that $\tilde{R} \geq R_0$ and $D_{\mu_x, p}(\tilde{R}) = D$. Therefore $R_{\mu_x, p}(D) = \tilde{R} \geq R_0 = R_{\mu, p}(D)$. Since $\mu_x$ arises from the ergodic decomposition of $\mu$, we can assume $\supp(\mu_x) \subset \supp(\mu)\subset \mS$.
\end{proof}

\begin{proof} \emph{(of Theorem \ref{thm:ergodic_var_prin}).}
	In order to prove the theorem, it suffices to show that for every $\eps>0$ and every $\mu \in \mP_{T}(\mX)$ there exists $\nu \in \mE_T(\mX)$ such that $R_{\mu,p}(\eps) \leq R_{\nu,p}(\eps)$. By Lemma \ref{lem:rate_distortion_ergodic}, this is the case for subshifts $\mS \subset A^\Z$ for the rate-distortion function $\tilde{R}_{\mu,p}$ from Definition \ref{def:rate_distortion_function}. We will extend this to the case of any topological dynamical system $(\mX, \rho, T)$ for the rate-distortion function $R_{\mu,p}$ from Definition \ref{def:rate_dist_LT}.\\
	Define the orbit map $\Phi : \mX \to \mX^\Z$ as $\Phi(x) = (T^n x)_{n \in \Z}$. It provides an equivariant homeomorphism between $(\mX, T)$ and $(\Phi(\mX), \sigma)$, where $\mX^\Z$ is considered with the product topology.
	Fix an ergodic measure $\mu \in \mE_{T}(\mX)$. Then the push-forward measure $\Phi_*\mu$ belongs to $\mE_{\sigma}(\Phi(\mX))$. Moreover, we see from Definitions \ref{def:rate_dist_LT}  and \ref{def:rate_distortion_function} that for $\eps>0$ we have
	\[ R_{\mu,p}(\eps) = \tilde{R}_{\Phi_* \mu,p}(\eps), \]
	where for $R_{\mu,p}(\eps)$ we treat $\rho$ as a metric on the phase space $\mX$ and for  $\tilde{R}_{\Phi_* \mu,p}(\eps)$ we treat $\rho$ as a metric on the alphabet space $\mX$ of the subshift $\Phi(\mX) \subset \mX^\Z$. According to the Lemma \ref{lem:rate_distortion_ergodic}, there exists a measure $\tilde{\nu} \in \mE_{\sigma}(\Phi(\mX))$ with
	\[ \tilde{R}_{\Phi_* \mu,p}(\eps) \leq \tilde{R}_{\tilde{\nu},p}(\eps).  \]
	Since $\supp(\tilde{\nu}) \subset \Phi(\mX)$, we can consider its push-back $\nu := (\Phi^{-1})_* \tilde{\nu}$, which satisfies
	\[ R_{\nu, p}(\eps) = \tilde{R}_{\tilde{\nu},p}(\eps) \geq \tilde{R}_{\Phi_* \mu,p}(\eps) = R_{\mu,p}(\eps).  \]
	Observing that $\nu \in \mE_{T}(\mX)$ concludes the proof.
\end{proof}

Theorem \ref{thm:ergodic_var_prin} was essentially proved in \cite[Proof of Proposition 7]{velozo2017rate}, although not explicitly
stated. Our approach is more direct, as it avoids the use of Katok's entropy (yet it is similar in spirit).

\begin{problem}
	Lindenstrauss and Tsukamoto proved a version of variational principle without the assumption of the tame growth of covering numbers, in terms of the $L^\infty$ rate-distortion function. Is Theorem \ref{thm:ergodic_var_prin} true also in this case?
\end{problem}

\section{R\'enyi information dimension and metric mean dimension}\label{sec:mrid}

Pesin defined the (upper) R\'{e}nyi information dimension of a Borel measure on a metric space $(\mX, \rho)$ as
\begin{equation}\label{eq:rid_pesin} \overline{\mathrm{RID}}(\mu) = \limsup \limits_{\eps \to 0} \frac{1}{\log \frac{1}{\eps}} \inf_{\diam(P)\leq\eps}H_{\mu}(P),
\end{equation}
where the infimum is taken over all Borel partitions of $\mX$ of diameter at most $\eps$, see \cite[p. 186]{pesin2008dimension} (see also \cite{renyi1959dimension} for the original definition in the context of random variables). Inspired by this definition and Theorem \ref{thm:lin_tsu_var_prin}, we prove the following version of the variational principle for metric mean dimension, given directly in terms of measure-theoretic entropy. Note that
the assumption of tame growth of covering numbers is not required.

\begin{thm}\label{thm:mrid_mdim} Let $(\mX, \rho, T)$ be a topological dynamical system. Then 
	\begin{equation}\label{eq:mrid_def_up}  \ummdim(\mX, \rho, T) = \limsup_{\eps\rightarrow0}\frac{1}{\log \frac{1}{\eps}}\ \sup_{\mu\in \mP_{T}(\mathcal{X})}\ \inf_{\diam(P)\leq\eps}h_{\mu}(P)\end{equation}
	and
	\begin{equation}\label{eq:mrid_def_low} \lmmdim(\mX, \rho, T) = \liminf_{\eps\rightarrow0}\frac{1}{\log\frac{1}{\eps}}\ \sup_{\mu\in \mP_{T}(\mathcal{X})}\ \inf_{\diam(P)\leq\eps}h_{\mu}(P),
	\end{equation}
	where the infima are taken over all Borel partitions $P$ of $\mX$ with diameter at most $\eps$.
\end{thm}

The proof is based on the version of the local variational principle for the entropy of a fixed open cover, conjectured by Romagnoli in \cite{Romagnoli03} and proved by Glasner and Weiss in \cite{GW06}.

\begin{defn}
	Let $(\mX, \rho, T)$ be a topological dynamical system and let $\mU$ be a finite open cover of $\mX$. The \emph{topological entropy of $\mU$} is
	\[ h_{\mathrm{top}}(\mU, T) = \lim \limits_{n \to \infty} \frac{1}{n} \log N(\mU^n), \]
	where $\mU^n = \mU \vee T^{-1}\mU \vee \ldots \vee T^{-(n-1)}\mU$ and $N(\mathcal{V})$ denotes the minimal cardinality of a subcover of a cover $\mathcal{V}$.
\end{defn}

\begin{thm}\cite[Theorem 7.11]{GW06}\label{thm:local_var_prin}
	Let $(\mX, \rho, T)$ be a topological dynamical system and let $\mU$ be a finite open cover of $\mX$. Then
	\[ \sup \limits_{\mu \in \mP_T(\mX)} \inf \limits_{P \succcurlyeq \mU} h_{\mu}(P) = h_{\mathrm{top}}(\mU, T), \]
	where the infimum is taken over all finite Borel partitions $P$ of $\mX$ which refine $\mU$ (i.e. every $A \in P$ is contained in an element of $\mU$).
\end{thm}

In order to connect the above local variational principle for open covers with the definitions of mean R\'enyi information dimension and metric mean dimension, which are given in terms of entropy and complexity in scale $\eps$, we need the following two lemmas.

By $\diam(\mU)$ we will denote the diamater of the cover $\mU$, i.e. the maximal diameter of an element of $\mU$. Recall that the Lebesgue number of an open cover $\mU$ of a metric space $(\mX, \rho)$ is the largest number $\eps$ with the property that every open ball of radius $\eps$ is contained in an element of $\mU$. We will denote the Lebesgue number of $\mU$ by $\Leb(\mU)$. A set $A \subset \mX$ is called an $\eps$-net if for every $y \in \mX$ there exists $x \in A$ such that $\rho(x,y) < \eps$.

\begin{lem}\label{lem:lebesgue_number}
	Let $(\mX, \rho)$ be a compact metric space. Then for every $\eps>0$ there exists a finite open cover $\mU$ of $\mX$ such that $\diam(\mU) \leq \eps$ and $\Leb(\mU) \geq \frac{\eps}{4}$.
\end{lem}

\begin{proof}
	Let $A \subset \mX$ be a finite $\frac{\eps}{4}$-net in $\mX$ (it exists due to the compactness of $\mX$). Let $\mathcal{U} = \{ B(x,\frac{\eps}{2}) : x \in A \}$, where $B(\cdot, \cdot)$ denotes an open ball in metric $\rho$. Then $\mU$ is an open cover of $\mX$ with $\diam(\mU) \leq \eps$. We shall prove that for every $y \in \mX$, the ball $B(y, \frac{\eps}{4})$ is contained in an element of $\mU$. For every $y \in \mX$ there exists $x \in A$ such that $\rho(x,y) < \frac{\eps}{4}$. Then $B(y, \frac{\eps}{4}) \subset B(x, \frac{\eps}{2})$ and $B(x, \frac{\eps}{2}) \in \mU$.
\end{proof}

\begin{lem}\label{lem:cover_res_entropies}
	Let $(\mX, \rho, T)$ be a topological 
	dynamical system and let $\mU$ be a finite open cover of $\mX$. Then
	\[ S(\mX, \rho, T, \diam(\mU)) \leq h_{\mathrm{top}}(\mU, T) \leq  S(\mX, \rho, T, \Leb(\mU)). \]
\end{lem}

\begin{proof}
	This follows from the following easy inequalities
	\[ \#(\mX, \rho_n, \eps) \leq N(\mU^n)\ \ \text{ for } \diam(\mU) \leq \eps \]
	and
	\[ N(\mU^n) \leq  \#(\mX, \rho_n, \eps)\ \  \text{ for } \eps \leq \Leb(\mU) \]
	(which can be derived e.g. as in the proof of Theorem 6.1.8 in \cite{D11}).
\end{proof}

\begin{proof}(of Theorem \ref{thm:mrid_mdim})
	Fix $\eps > 0$ and a finite open cover $\mU$ of $\mX$ such that $\diam(\mU) \leq \eps$ and $\Leb(\mU) \geq \frac{\eps}{4}$ (it exists due to Lemma \ref{lem:lebesgue_number}). As any Borel partition $P$ of $\mX$ which refines $\mU$ must be of diameter at most $\eps$, applying Theorem \ref{thm:local_var_prin} and Lemma \ref{lem:cover_res_entropies} we obtain
	\begin{equation}
	\begin{gathered}\label{eq:eps_mrid_leq_mmdim} \sup \limits_{\mu \in \mP_T(\mX)} \inf \limits_{\diam(P) \leq \eps} h_{\mu}(P) \leq  \sup \limits_{\mu \in \mP_T(\mX)} \inf \limits_{P \succcurlyeq \mU} h_{\mu}(P) = h_{\mathrm{top}}(\mU, T) \leq \\ \leq S(\mX, \rho, T, \Leb(\mU)) \leq S(\mX, \rho, T, \frac{\eps}{4}).
	\end{gathered}
	\end{equation}
	
	For the opposite inequality, note that if $P$ is any partition of $\mX$ with $\diam(P) \leq \frac{\eps}{8}$, then $P$ refines $\mU$ as $\Leb(\mU) \geq \frac{\eps}{4}$. Therefore applying again  Theorem \ref{thm:local_var_prin} and Lemma \ref{lem:cover_res_entropies} we get
	
	\begin{equation}
	\begin{gathered}\label{eq:eps_mrid_geq_mmdim}
	S(\mX, \rho, T, \eps) \leq S(\mX, \rho, T, \diam(\mU)) \leq h_{\mathrm{top}}(\mU, T) = \\
	= \sup \limits_{\mu \in \mP_T(\mX)} \inf \limits_{P \succcurlyeq \mU} h_{\mu}(P) \leq \sup \limits_{\mu \in \mP_T(\mX)} \inf \limits_{\diam(P) \leq \frac{\eps}{8}} h_{\mu}(P).
	\end{gathered}
	\end{equation}
	
	Dividing both sides of (\ref{eq:eps_mrid_leq_mmdim}) and (\ref{eq:eps_mrid_geq_mmdim}) by $\log \frac{1}{\eps}$ and taking the limits yields the desired inequalities (note that $\lim \limits_{\eps \to 0}\frac{\log \frac{1}{\eps}}{\log C \frac{1}{\eps}} = 1$ for any $C>0$).
\end{proof}

\begin{rem}
	It follows from \cite[Proposition 5]{HMRY04} that in Theorem \ref{thm:local_var_prin} it suffices to take supremum over the set $\mE_T(\mX)$ of ergodic measures rather than $\mP_T(\mX)$. Consequently, the same is true for the variational principle of Theorem \ref{thm:mrid_mdim}.
\end{rem}

\section{Mean R\'enyi information dimension of a measure}\label{sec:single_mrid}

Following (\ref{eq:rid_pesin}), we introduce the upper and lower \emph{mean R\'enyi information dimensions} of an invariant measure $\mu \in \mP_{T}(\mX)$ as:
\[\overline{MRID}(\mathcal{X},T,\mu,d)=\limsup_{\eps\rightarrow0}\frac{1}{\log\frac{1}{\eps}}\inf_{\diam(P)\leq\eps}h_{\mu}(P)\]
and
\[\underline{MRID}(\mathcal{X},T,\mu,d)=\liminf_{\eps\rightarrow0}\frac{1}{\log\frac{1}{\eps}}\inf_{\diam(P)\leq\eps}h_{\mu}(P),\]
where the infimum is taken over all Borel partitions $P$ of $\mX$ with diameter not greater than $\eps$.

Geiger and Koch proposed recently the following similar definition for stationary stochastic processes with values in $\R$. Let $\mP_{\sigma}(\R^\Z)$ denote the set of shift-invariant Borel probability measures on $\R^\Z$.
\begin{defn}
	For $\mu \in \mP_{\sigma}(\R^\Z)$, the lower and upper \emph{information dimension rates} are defined as
	\[ \overline{d}(\mu) = \limsup \limits_{m \to \infty} \lim \limits_{n \to \infty} \frac{H_{\mu}(\bigvee_{i=0}^{n-1}\sigma^{-i} P_{m})}{n \log m} = \limsup \limits_{m \to \infty} \frac{h_\mu(P_m)}{\log m}, \]
	\[ \underline{d}(\mu) = \liminf \limits_{m \to \infty} \lim \limits_{n \to \infty} \frac{H_{\mu}(\bigvee_{i=0}^{n-1}\sigma^{-i} P_{m})}{n \log m} = \liminf \limits_{m \to \infty} \frac{h_\mu(P_m)}{\log m},\]
	where $P_m := \pi_0^{-1}\big( \{[\frac{i}{m}, \frac{i+1}{m}) : m \in \Z \}\big)$.
\end{defn}
Note that if $H_\mu(P_1) = \infty$, then $\overline{d}(\mu) = \underline{d}(\mu) = \infty$. Geiger and Koch proved (\cite[Theorem 1]{GK17} and \cite[Theorem 18]{GK18}) that
\begin{equation}\label{e:geiger_koch_rate} \overline{d}(\mu) = \overline{\dim}_{R, 2} (\mu) \text{ and } \underline{d}(\mu) = \underline{\dim}_{R, 2} (\mu) \text{ for } \mu \in \mP_{\sigma}([0,1]^{\Z}),
\end{equation}
where $\underline{\dim}_{R, 2} (\mu)$ and $\overline{\dim}_{R, 2}$ are the lower and upper \emph{rate-distortion dimensions} given as
\begin{equation}\label{eq:rdim_def}
\begin{gathered}
\overline{\dim}_{R, 2}(\mu) = \limsup \limits_{\eps \to 0} \lim \limits_{n \to \infty} \frac{\tilde{R}_{\mu, 2}(n, \eps)}{\log \frac{1}{\eps}} = \limsup \limits_{\eps \to 0} \frac{\tilde{R}_{\mu, 2}(\eps)}{\log \frac{1}{\eps}}, \\ \underline{\dim}_{R, 2}(\mu) = \liminf \limits_{\eps \to 0} \lim \limits_{n \to \infty} \frac{\tilde{R}_{\mu, 2}(n, \eps)}{\log \frac{1}{\eps}} = \liminf \limits_{\eps \to 0} \frac{\tilde{R}_{\mu, 2}(\eps)}{\log \frac{1}{\eps}}.
\end{gathered}
\end{equation}
This can be seen as a generalization of \cite[Proposition 3.3]{kawabata1994rate} from the Bernoulli case to general stationary processes. The following proposition states that mean R\'enyi information dimension and information dimension rate coincide for ergodic shift-invariant measures corresponding to almost surely bounded processes.
\begin{prop}\label{prop:mrid_d_subshifts}
	Let $\mu \in \mE_{\sigma}([0,1]^\Z)$. Then
	\[ \overline{MRID}([0,1]^\Z, \sigma,\mu,\tau) = \overline{d}(\mu) \text{ and } \underline{MRID}([0,1]^\Z, \sigma,\mu,\tau)  = \underline{d}(\mu)\]
\end{prop}
\begin{proof}
	It is clear from the definitions that $\overline{MRID}([0,1]^\Z, \sigma,\mu,\tau) \leq \overline{d}(\mu)$ and $\underline{MRID}([0,1]^\Z, \sigma,\mu,\tau)  \leq \underline{d}(\mu)$ for any $\mu \in \mP_{\sigma}([0,1]^\Z)$. By \cite[Proposition 2]{velozo2017rate}, if $\diam(P)<\eps$
	and $\mu$ is ergodic then $R_{\mu, 2}(\varepsilon)\leq h_{\mu}(P)$ (\cite[Proposition 2]{velozo2017rate} is proved for $L^1$-distortion function, however it is easy to see that it is valid also for $L^p,\ p \in [1, \infty)$). As $\tilde{R}_{\mu, p}(\eps) \leq R_{\mu,p} (\eps)$ (see Proposition \cite[Proposition C-B.1]{GS20mmdim_compress}), the result follows from (\ref{e:geiger_koch_rate}).
\end{proof}
\begin{problem}
	Is Proposition \ref{prop:mrid_d_subshifts} true for non-ergodic measures?
\end{problem}

\section{Mean Brin-Katok local entropy}\label{sec:brin_katok}
The goal of this section is to derive a novel application in ergodic theory which follows from the techniques of the article. Let $(\mX, \rho, T)$ be a topological dynamical system. Following \cite[Definition 6.7.1]{Dz2005},
for an ergodic measure $\mu\in\mathcal{{E}}_{T}(\mathcal{{X}})$ and a point $x \in \mX$ we define the \emph{Brin-Katok local entropy} by:
\[h_{\mu}^{BK}(\eps,x)=\limsup_{n \to \infty}-\frac{1}{n}\log\mu(B_{n}(x,\eps)),\]
where $B_{n}(x,\eps)=\{y\in\mathcal{{X}} : \rho_{n}(x,y)<\eps\}$.
It is easy to see that by the ergodicity of $\mu$, the function $x \mapsto h_{\mu}^{BK}(\eps,x)$ is $\mu$-almost surely constant and we denote this constant by $h_{\mu}^{BK}(\eps)$. Let us define the \emph{mean Brin-Katok local entropy} of $T$ by:
\[\overline{\mBKe}(\mX, \rho, T)=\limsup_{\eps\rightarrow0}\frac{1}{\log\frac{1}{\eps}}\sup_{\mu\in\mE_{T}(\mX)}h_{\mu}^{BK}(\eps)\]
and
\[\underline{\mBKe}(\mX, \rho, T)=\liminf_{\eps\rightarrow0}\frac{1}{\log \frac{1}{\eps}}\sup_{\mu\in\mE_{T}(\mX)}h_{\mu}^{BK}(\eps).\]
\begin{lem}\label{l:brin_katok}
	Let $(\mX, \rho, T)$ be a topological dynamical system. For $\mu\in\mathcal{E}_{T}(\mathcal{X})$
	and $\varepsilon>0$ it holds that \[h_{\mu}^{BK}(\eps)\leq\inf_{\diam(P)<\eps}h_{\mu}(P),\]
	where the infimum is taken over all finite Borel partitions of $\mX$ with diameter smaller than $\eps.$
\end{lem}
The proof of Lemma \ref{l:brin_katok} uses Shannon-McMillan-Breiman theorem and can be essentially found in
\cite{BK83}. We include it for the convenience of the reader.
\begin{proof} For a partition $P$ of $\mathcal{X}$, denote by $P(x)$ the element
	of $P$ containing $x \in \mX$. If $\diam(P)<\varepsilon$, then $(\bigvee \limits_{i=0}^{n-1} T^{-i}P)(x)\subset B_{n}(x,\varepsilon)$,
	hence
	\[-\log\mu(B_{n}(x,\varepsilon))\leq-\log\mu\Big((\bigvee \limits_{i=0}^{n-1} T^{-i}P)(x)\Big).\]
	By the Shannon-McMillan-Breiman theorem (\cite[Theorem I.5.1]{shields1996ergodic})
	\[h_{\mu}^{BK}(\eps,x)=\limsup \limits_{n \to \infty}-\frac{1}{n}\log\mu(B_{n}(x,\eps))\leq  \limsup\limits_{n\to\infty}-\frac{1}{n}\log\mu\big((\bigvee \limits_{i=0}^{n-1}T^{-i}P)(x)\big)=h_{\mu}(P)\]
	for $\mu$-a.e. $x\in\mathcal{X}$, hence $h_{\mu}^{BK}(\eps)\leq\inf \limits_{\diam(P)<\eps}h_{\mu}(P)$.
\end{proof} Applying Theorem \ref{thm:mrid_mdim} we obtain the following corollary:

\begin{cor}\label{cor:mbke_mrid_ineq} Let $(\mX, \rho, T)$ be a topological dynamical system. Then \[\overline{mBKe}(\mX, \rho, T)\leq \ummdim(\mX, \rho, T)\]
	and
	\[\underline{mBKe}(\mX, \rho, T)\leq \lmmdim(\mX, \rho, T).\]
\end{cor}

\begin{problem}\label{problem3}\footnote{This question was recently answered affirmatively by Ruxi Shi in \cite{Shi21}.}
	Does equality hold in Corollary \ref{cor:mbke_mrid_ineq}?
\end{problem}

By Corollary \ref{cor:mbke_mrid_ineq} we have the following somewhat surprising proposition (note that $\eps_0$ is chosen uniformly for $\mu \in \mE_T(\mX)$).

\begin{prop}\label{prop:mbke_f_ball_bound}
	Let $(\mX, \rho, T)$ be a topological dynamical system with $\ummdim(\mX, \rho, T) < \infty$. For every $\delta > 0$ there exists $\eps_0 > 0$ such that for \textit{every} $0 < \eps < \eps_0$ and \textit{every} ergodic measure $\mu \in \mE_T(\mX)$, for $\mu$-almost every point $x \in \mX$ there exists $n_0 \in \N$ such that for every $n \geq n_0$ the inequality
	\[\mu(B_n(x,\eps)) \geq \eps^{n(\ummdim(\mX, \rho, T) + \delta)}\]
	holds.
\end{prop}

\bibliographystyle{alpha}
\bibliography{universal_bib}

\end{document}